\documentclass[10pt]{amsart}
\usepackage{amssymb,amsthm,amsmath,amsfonts}
\usepackage{natbib}
\usepackage{hyperref}
\usepackage{enumerate}
\usepackage{color}

\makeatletter
\g@addto@macro\th@plain{\thm@headpunct{}}
\makeatother

\newtheorem{theorem}{Theorem}[section]

\newtheorem{corollary}[theorem]{Corollary}
\newtheorem{lemma}[theorem]{Lemma}
\newtheorem{remark}[theorem]{Remark}

\newcommand{\xx}{ {\textbf x} }
\newcommand{\ab}{ {\textbf a} }
\newcommand{\bb}{ {\textbf b} }
\newcommand{\cb}{ {\textbf c} }
\newcommand{\dd}{ {\textbf d} }
\newcommand{\yy}{ {\textbf y} }

\newcommand{\zz}{ {\textbf z} }
\newcommand{\ee}{ {\textbf e} }

\newcommand{\ub}{ {\textbf u} }

\newcommand{\cbo}{ {\textbf c}^\bot }

\newcommand{\VV}{ \Omega }
\newcommand{\RR}{\mathbb{R}}
\newcommand{\KK}{\mathbb{K}}

\newcommand{\LL}{\mathbb{L}}
\newcommand{\PP}{\mathbb{P}}
\newcommand{\En}{\mathbb{E}}

\newcommand{\TT}{\mathcal{T}}

\newcommand{\tr}{\mathrm{tr}\,}
\newcommand{\Trace}{\mathrm{Trace}\,}
\newcommand{\DDet}{\mathrm{Det}}

\providecommand{\norm}[1]{\lVert#1\rVert}

\providecommand{\scalar}[1]{\left\langle#1\right\rangle}

\title{Multiplicative Cauchy functional equation on symmetric cones}
\author[B. Ko\l{}odziejek]{Bartosz Ko\l{}odziejek}
\address{Faculty of Mathematics and Information Science\\Warsaw University of Technology\\Pl. Politechniki 1\\00-661 Warszawa, Poland}
\email{kolodziejekb@mini.pw.edu.pl}



\begin{document}

\begin{abstract}
We solve the logarithmic Cauchy functional equation in the symmetric cone with respect to two different multiplication algorithms. We impose no regularity assumptions on respective functions. 
\end{abstract}
\maketitle

\section{Introduction}
The multiplicative Cauchy functional equation in the space of real square matrices of the form
\begin{align*}
f(\xx)f(\yy)=f(\xx\cdot \yy),
\end{align*}
where $\xx$ and $\yy$ are matrices and $f$ is real-valued function, was studied by \citet{Golab}, and other contributors \citet{Hosszu, Kuczma1959, KuchZaj, Djoko1970}. It is known that the general solution is of the form
\begin{align*}
f(\xx)=g(\DDet(\xx)),
\end{align*}
where $g$ is a multiplicative function, that is $g(a)g(b)=g(ab)$ for any positive $a$ and $b$. The cone of symmetric positive definite matrices $\VV_+$ is not closed under multiplication $\xx\cdot\yy$, therefore it is common to consider on $\VV_+$ multiplicative functions of the form
\begin{align*}
f(\xx)f(\yy)=f(\xx^{1/2}\cdot \yy\cdot \xx^{1/2}),\quad (\xx,\yy)\in\VV_+^2.
\end{align*}
Multiplication $\xx^{1/2}\cdot \yy\cdot \xx^{1/2}$ in the cone $\VV_+$ is closely related to Jordan triple product $(\xx,\yy)\mapsto\xx\cdot\yy\cdot\xx$ and can be naturally replaced by more abstract operations. In this paper we are interested in finding all real functions of symmetric positive definite matrices $\VV_+$, satisfying
\begin{align}\label{wC}
f(\xx)+f(w_I\cdot \yy\cdot w_I^T)=f(w_\xx\cdot \yy \cdot w_\xx^T),\quad (\xx,\yy)\in\VV_+^2,
\end{align}
where $w_\xx$ is a measurable map such that, $\xx=w_\xx\cdot w_\xx^T$ for any $\xx\in\VV_+$, and $I$ is the identity matrix. Such maps are called multiplication algorithms. 

Our interest in this functional equation stems from investigations of characterization problems for probabilistic measures concentrated on $\VV_+$ or more generally on symmetric cones - see \citet{BW2002}, \citet{HaLaZi2008}, \citet{Bo2005,Bout2009}, \citet{BK2010,BK2013}.

We find all functions satisfying \eqref{wC} with respect to two basic multiplication algorithms. First one is connected to Jordan triple product: $w^{(1)}_\xx=\xx^{1/2}$, where $\xx^{1/2}$ is the unique symmetric positive definite square root of $\xx=\xx^{1/2}\cdot\xx^{1/2}$. The second is $w^{(2)}_\xx=t_\xx$, where $t_\xx$ is the lower triangular matrix in the Cholesky decomposition of $\xx$. We impose no regularity assumptions on $f$. This problem is naturally generalized to irreducible symmetric cones setting.

Functional equations for $w^{(1)}$ were already considered in \citet{BW2003} for differentiable functions and in \citet{Molnar2006} for continuous functions of real or complex Hermitian positive definite matrices of rank strictly greater than $2$. Without any regularity assumptions it was solved on the Lorentz cone \citet{Wes2007L}. Molnar is his paper wrote: ``However, we suspect that the same conclusions are valid also when $r=2$. It would be interesting to find a proof for this case which would probably give a new approach to the case $r>2$ as well without invoking Gleason's theorem''. Such proof, for any of the five types of irreducible symmetric cones, is given in Section 3. Moreover, we do not assume continuity of respective functions. Our approach is through Peirce decomposition on symmetric cones. In Section 4 we formulate main theorems in the language of the less known Lorentz cone framework.

Case of $w^{(2)}$, perhaps a bit surprisingly, leads to a different solution. It was indirectly solved for differentiable functions in \citet{HaLaZi2008}. Here we solve this functional equation under no regularity assumptions and, at the same time, we find all real characters of triangular group.

The problem of solving \eqref{wC} for any multiplication algorithm $w$ remains a challenge.

\section{Preliminaries}
In this section we give a short introduction to the theory of symmetric cones. For further details we refer to \citet{FaKo1994}. 

A \textit{Euclidean Jordan algebra} is a Euclidean space $\En$ (endowed with scalar product denoted $\scalar{\xx,\yy}$) equipped with a bilinear mapping (product)
\begin{align*}
\En\times\En \ni \left(\xx,\yy\right)\mapsto \xx\yy\in\En
\end{align*}
and a neutral element $\ee$ in $\En$ such that for all $\xx$, $\yy$, $\zz$ in $\En$:
\begin{enumerate}[(i)]
	\item $\xx\yy=\yy\xx$, 
	\item $\xx(\xx^2\yy)=\xx^2(\xx\yy)$,
	\item $\xx\ee=\xx$,
	\item $\scalar{\xx,\yy\zz}=\scalar{\xx\yy,\zz}$.
\end{enumerate}
For $\xx\in\En$ let $\LL(\xx)\colon \En\to\En$ be linear map defined by
\begin{align*}
\LL(\xx)\yy=\xx\yy,
\end{align*}
and define 
\begin{align*}
\PP(\xx)=2\LL^2(\xx)-\LL\left(\xx^2\right).
\end{align*} 
The map $\PP\colon \En\mapsto End(\En)$ is called the \emph{quadratic representation} of $\En$.

An element $\xx$ is said to be \emph{invertible} if there exists an element $\yy$ in $\En$ such that $\LL(\xx)\yy=\ee$. Then $\yy$ is called the \emph{inverse of} $\xx$ and is denoted by $\yy=\xx^{-1}$. Note that the inverse of $\xx$ is unique. It can be shown that $\xx$ is invertible if and only if $\PP(\xx)$ is invertible and in this case $\left(\PP(\xx)\right)^{-1} =\PP\left(\xx^{-1}\right)$.

Euclidean Jordan algebra $\En$ is said to be \emph{simple} if it is not a \mbox{Cartesian} product of two Euclidean Jordan algebras of positive dimensions. Up to linear isomorphism there are only five kinds of Euclidean simple Jordan algebras. Let $\mathbb{K}$ denote either the real numbers $\RR$, the complex ones $\mathbb{C}$, quaternions $\mathbb{H}$ or the octonions $\mathbb{O}$, and write $S_r(\mathbb{K})$ for the space of $r\times r$ Hermitian matrices valued in $\mathbb{K}$, endowed with the Euclidean structure $\scalar{\xx,\yy}=\Trace(\xx\cdot\bar{\yy})$ and with the Jordan product
\begin{align}\label{defL}
\xx\yy=\tfrac{1}{2}(\xx\cdot\yy+\yy\cdot\xx),
\end{align}
where $\xx\cdot\yy$ denotes the ordinary product of matrices and $\bar{\yy}$ is the conjugate of $\yy$. Then $S_r(\RR)$, $r\geq 1$, $S_r(\mathbb{C})$, $r\geq 2$, $S_r(\mathbb{H})$, $r\geq 2$, and the exceptional $S_3(\mathbb{O})$ are the first four kinds of Euclidean simple Jordan algebras. Note that in this case 
\begin{align}\label{defP}
\PP(\yy)\xx=\yy\cdot\xx\cdot\yy.
\end{align}
The fifth kind is the Euclidean space $\RR^{n+1}$, $n\geq 2$, with Jordan product
\begin{align}\label{scL}\begin{split}
\left(x_0,x_1,\dots, x_n\right)\left(y_0,y_1,\dots,y_n\right) =\left(\sum_{i=0}^n x_i y_i,x_0y_1+y_0x_1,\dots,x_0y_n+y_0x_n\right).
\end{split}
\end{align}

To each Euclidean simple Jordan algebra one can attach the set of Jordan squares
\begin{align*}
\bar{\VV}=\left\{\xx^2\colon\xx\in\En \right\}.
\end{align*}
The interior $\VV$ is called a symmetric cone.
Moreover $\VV$ is \emph{irreducible}, i.e. it is not the Cartesian product of two convex cones. One can prove that an open convex cone is symmetric and irreducible if and only if it is the cone $\VV$ of some Euclidean simple Jordan algebra. Each simple Jordan algebra corresponds to a symmetric cone, hence there exist up to linear isomorphism also only five kinds of symmetric cones. The cone corresponding to the Euclidean Jordan algebra $\RR^{n+1}$ equipped with Jordan product \eqref{scL} is called the Lorentz cone. 

We denote by $G(\En)$ the subgroup of the linear group $GL(\En)$ of linear automorphisms which preserves $\VV$, and we denote by $G$ the connected component of $G(\En)$ containing the identity.  Recall that if $\En=S_r(\RR)$ and $GL(r,\RR)$ is the group of invertible $r\times r$ matrices, elements of $G(\En)$ are the maps $g\colon\En\to\En$ such that there exists $\ab\in GL(r,\RR)$ with
\begin{align*}
g(\xx)=\ab\cdot\xx\cdot\ab^T.
\end{align*}
We define $K=G\cap O(\En)$, where $O(\En)$ is the orthogonal group of $\En$. It can be shown that 
\begin{align*}
K=\{ k\in G\colon k\ee=\ee \}.
\end{align*} 

A \emph{multiplication algorithm} is a measurable map $\VV\to G\colon \xx\mapsto w(\xx)$ such that $w(\xx)\ee=\xx$ for all $\xx$ in $\VV$. This concept is consistent with, so called, division algorithm, which was introduced in \citet{OlRu1962} and \citet{CaLe1996}. One of two important examples of multiplication algorithm is the map $w_1(\xx)=\PP\left(\xx^{1/2}\right)$.

We will now introduce a very useful decomposition in $\En$, called \emph{spectral decomposition}. An element $\cb\in\En$ is said to be a \emph{primitive idempotent} if $\cb\cb=\cb\neq 0$ and if $\cb$ is not a sum of two non-null idempotents. A \emph{complete system of primitive orthogonal idempotents} is a set $\left(\cb_1,\dots,\cb_r\right)$ such that
\begin{align*}
\sum_{i=1}^r \cb_i=\ee\quad\mbox{and}\quad\cb_i\cb_j=\delta_{ij}\cb_i\quad\mbox{for } 1\leq i\leq j\leq r.
\end{align*}
The size $r$ of such system is a constant called the \emph{rank} of $\En$. Any element $\xx$ of a Euclidean simple Jordan algebra can be written as $\xx=\sum_{i=1}^r\lambda_i\cb_i$ for some complete $\left(\cb_1,\dots,\cb_r\right)$ system of primitive orthogonal idempotents. The real numbers $\lambda_i$, $i=1,\dots,r$ are the \emph{eigenvalues} of $\xx$. One can then define \emph{trace} and \emph{determinant} of $\xx$ by, respectively, $\tr\xx=\sum_{i=1}^r\lambda_i$ and $\det\xx=\prod_{i=1}^r\lambda_i$. An element $\xx\in\En$ belongs to $\VV$ if and only if all its eigenvalues are strictly positive. 

The rank $r$ and $\dim\VV$ of irreducible symmetric cone are connected through relation
\begin{align*}
\dim\VV=r+\frac{d r(r-1)}{2},
\end{align*}
where $d$ is an integer called the \emph{Peirce constant}. 

If $\cb$ is a primitive idempotent of $\En$, the only possible eigenvalues of $\LL(\cb)$ are $0$, $\tfrac{1}{2}$ and $1$. We denote by $\En(\cb,0)$, $\En(\cb,\tfrac{1}{2})$ and $\En(\cb,1)$ the corresponding eigenspaces. The decomposition 
\begin{align*}
\En=\En(\cb,0)\oplus\En(\cb,\tfrac{1}{2})\oplus\En(\cb,1)
\end{align*}
is called the \emph{Peirce decomposition of $\En$ with respect to $\cb$}. Note that $\PP(\cb)$ is the orthogonal projection of $\En$ onto $\En(\cb,1)$.

Fix a complete system of orthogonal idempotents $\left(\cb_i\right)_{i=1}^r$. Then for any $i,j\in\left\{1,2,\dots,r\right\}$ we write
\begin{align*}
\begin{split}
\En_{ii} & =\En(\cb_i,1)=\RR \cb_i, \\
\En_{ij} & = \En\left(\cb_i,\frac{1}{2}\right) \cap \En\left(\cb_j,\frac{1}{2}\right) \mbox{ if }i\neq j.
\end{split}
\end{align*}
It can be proved (see \cite[Theorem IV.2.1]{FaKo1994}) that
\begin{align*}
\En=\bigoplus_{i\leq j}\En_{ij}
\end{align*}
and (symbol $``\cdot"$ denotes here the Jordan product)
\begin{align*}
\begin{split}
\En_{ij}\cdot\En_{ij} & \subset\En_{ii}+\En_{ij}, \\
\En_{ij}\cdot\En_{jk} & \subset\En_{ik},\mbox{ if }i\neq k, \\
\En_{ij}\cdot\En_{kl} & =\{0\},\mbox{ if }\{i,j\}\cap\{k,l\}=\emptyset.
\end{split}\end{align*}
 Moreover (\cite[Lemma IV.2.2]{FaKo1994}), if $\xx\in\En_{ij}$, $\yy\in\En_{jk}$, $i\neq k$, then 
\begin{align}\label{xxyy}
\xx^2 & =\tfrac{1}{2}\norm{\xx}^2(\cb_i+\cb_j),\\
\notag\norm{\xx\yy}^2 & =\tfrac{1}{8}\norm{\xx}^2\norm{\yy}^2.
\end{align}
The dimension of $\En_{ij}$ is the Peirce constant $d$ for any $i\neq j$. When $\En$ is $S_r(\KK)$, if $(e_1,\dots,e_r)$ is an orthonormal basis of $\RR^r$, then $\En_{ii}=\RR e_i e_i^T$ and $\En_{ij}=\KK\left(e_i e_j^T+e_je_i^T\right)$ for $i<j$ and $d$ is equal to $dim_{|\RR}\KK$.
	
For $1\leq k \leq r$ let $P_k$ be the orthogonal projection onto $\En^{(k)}=\En(\cb_1+\ldots+\cb_k,1)$, $\det^{(k)}$ the determinant in the subalgebra $\En^{(k)}$, and, for $\xx\in\VV$, $\Delta_k(\xx)=\det^{(k)}(P_k(\xx))$. Then $\Delta_k$ is called the principal minor of order $k$ with respect to the Jordan frame $(\cb_k)_{k=1}^r$. Note that $\Delta_r(\xx)=\det\xx$. For $s=(s_1,\ldots,s_r)\in\RR^r$ and $\xx\in\VV$, we write
\begin{align*}
\Delta_s(\xx)=\Delta_1(\xx)^{s_1-s_2}\Delta_2(\xx)^{s_2-s_3}\ldots\Delta_r(\xx)^{s_r}.
\end{align*}
Note that symbol $\Delta_s$ is the same for $s$ being a number and a vector. Such notation was proposed in \citet{FaKo1994} and should not result in misleading ambiguity.
If $\xx=\sum_{i=1}^r\alpha_i\cb_i$, then $\Delta_s(\xx)=\alpha_1^{s_1}\alpha_2^{s_2}\ldots\alpha_r^{s_r}$. 

We will now introduce some basic facts about triangular group. For $\xx$ and $\yy$ in $\VV$, let $\xx\Box\yy$ denote the endomorphism of $\En$ defined by
\begin{align*}
\xx\Box\yy=\LL(\xx\yy)+\LL(\xx)\LL(\yy)-\LL(\yy)\LL(\xx).
\end{align*}
If $\cb$ is an idempotent and $\zz\in\En(\cb,\frac{1}{2})$ we define the \emph{Frobenius transformation $\tau_\cb(\zz)$ in $G$} by
\begin{align*}
\tau_\cb(\zz)=\exp(2\zz\Box\cb).
\end{align*}
Since $2\zz\Box\cb$ is nilpotent (see \cite[Lemma VI.3.1]{FaKo1994}) we get
\begin{align}\label{taucb}
\tau_{\cb}(\zz)=I+(2\zz\Box\cb)+\frac{1}{2}(2\zz\Box\cb)^2.
\end{align}
Given a Jordan frame $(\cb_i)_{i=1}^r$, the subgroup of $G$, 
\begin{align*}
\TT=\left\{\tau_{\cb_1}(\zz^{(1)})\ldots\tau_{\cb_{r-1}}(\zz^{(r-1)})\PP\left(\sum_{i=1}^r \alpha_i\cb_i\right)\colon \alpha_i>0, \zz^{(j)}\in \bigoplus_{k=j+1}^r\En_{jk}\right\}
\end{align*}
is called the \emph{triangular group corresponding to the Jordan frame $(\cb_i)_{i=1}^r$}. For any $\xx$ in $\VV$ there exists a unique $t_{\xx}$ in $\TT$ such that $\xx=t_{\xx}\ee$, that is, there exist (see \cite[Theorem IV.3.5]{FaKo1994}) elements $\zz^{(j)}\in \bigoplus_{k=j+1}^r \En_{jk}$, $1\leq j\leq r-1$ and positive numbers $\alpha_1, \ldots ,\alpha_r$ such that
\begin{align}\label{triangular}
\xx=\tau_{\cb_1}(\zz^{(1)})\tau_{\cb_2}(\zz^{(2)})\ldots\tau_{\cb_{r-1}}(\zz^{(r-1)})\left(\sum_{k=1}^r \alpha_k \cb_k \right).
\end{align}
Mapping $w_2\colon\VV\to\TT, \xx\mapsto w_2(\xx)=t_{\xx}$ is a multiplication algorithm.

For $\En=S_r(\RR)$ we have $\VV=\VV_+$. Let us define for $1\leq i,j\leq r$ matrix $\mu_{ij}=\left(\gamma_{kl}\right)_{1\leq k,l\leq r}$ such that $\gamma_{ij}=1$ and all other entries are equal $0$.
Then for Jordan frame $\left(\cb_i\right)_{i=1}^r$, where $\cb_k=\mu_{kk}$, $k=1,\ldots,r$, we have $\zz_{jk}=(\mu_{jk}+\mu_{kj})\in\En_{jk}$ and $\norm{\zz_{jk}}^2=2$, $1\leq j,k\leq r$, $j\neq k$.
If $\zz^{(i)}\in\bigoplus_{j=i+1}^r \En_{ij}$, $i=1,\ldots,r-1$, then there exists $\alpha^{(i)}=(\alpha_{i+1},\ldots,\alpha_r)\in\RR^{r-i}$ such that $\zz^{(i)}=\sum_{j=i+1}^r \alpha_j\zz_{ij}$. Then the Frobenius transformation reads
$$\tau_{\cb_i}(\zz^{(i)})\xx=\mathcal{F}_i(\alpha^{(i)})\cdot\xx\cdot \mathcal{F}_i(\alpha^{(i)})^T,$$
where $\mathcal{F}_i(\alpha^{(i)})$ is so called Frobenius matrix:
\begin{align*}
\mathcal{F}_i(\alpha^{(i)})=I+\sum_{j=i+1}^r \alpha_j \mu_{ji},
\end{align*}
ie. bellow $i$th one of identity matrix there is a vector $\alpha^{(i)}$, particularly
\begin{align*}
\mathcal{F}_2(\alpha^{(2)})=\begin{pmatrix}
  1    &   0    &   0    & \cdots & 0 \\
  0    &   1    &   0    & \cdots & 0 \\
  0    & \alpha_{3} &   1    & \cdots & 0 \\
\vdots & \vdots & \vdots & \ddots & \vdots \\
  0    & \alpha_{r} &   0    & \cdots & 1
\end{pmatrix}.
\end{align*}

It can be shown (\cite[Proposition VI.3.10]{FaKo1994}) that for each $t\in\TT$, $\xx\in\VV$ and $s\in\RR^r$,
\begin{align}\label{deltast}
\Delta_s(t\xx)=\Delta_s(t\ee)\Delta_s(\xx)
\end{align}
and for any $\zz\in\En(\cb_i,\frac{1}{2})$, $i=1,\ldots,r$,
\begin{align}\label{deltastau}
\Delta_s(\tau_{\cb_i}(\zz)\ee)=1,
\end{align}
if only $\Delta_s$ and $\TT$ are associated with the same Jordan frame $\left(\cb_i\right)_{i=1}^r$.

\section{Functional equations}
As was aforementioned, the problem \eqref{wC} is naturally generalized to symmetric cone setting. We are looking for functions $f\colon\VV\to\RR$ that satisfy following functional equation 
\begin{align}\label{wC2}
f(\xx)+f(w(\ee)\yy)=f(w(\xx)\yy),\quad (\xx,\yy)\in\VV^2.
\end{align}
Such functions we will call $w$-logarithmic Cauchy functions. By \cite[Proposition III.4.3]{FaKo1994}, for any $g$ in the group $G$
\begin{align*}
\det(g\xx)=(\DDet\,g)^{r/\dim\VV}\det\xx,
\end{align*}
where $\DDet$ denotes the determinant in the space of endomorphisms on $\VV$. Inserting a multiplication algorithm $g=g(\yy)$, $\yy\in\VV$, and $\xx=\ee$ we obtain
\begin{align*}
\DDet\left(w(\yy)\right) =(\det\yy)^{\dim\VV/r}
\end{align*}
and hence
\begin{align*}
\det(w(\yy)\xx) =\det\yy\det\xx
\end{align*}
for any $\xx,\yy\in\VV$. This means that $f(\xx)=H(\det\xx)$, where $H$ is generalized logarithmic function, ie. $H(ab)=H(a)+H(b)$ for $a,b>0$, is always a solution to \eqref{wC2}, regardless of multiplication algorithm $w$. If a $w$-logarithmic functions $f$ is additionally $K$-invariant ($f(\xx)=f(k\xx)$ for any $k\in K$), then $H(\det\xx)$ is the only possible solution (Theorem \ref{XXX}).

\begin{remark}
Consider a multiplication algorithm for which $K\ni w(\ee)\neq Id_\VV$.
Then equation \eqref{wC2} can be written as
\begin{align*}
f(\xx)+f(\yy)=f(\widetilde{w}(\xx)\yy),\quad (\xx,\yy)\in\VV^2,
\end{align*}
where $\widetilde{w}(\xx)=w(\xx)w^{-1}(\ee)$ and $\widetilde{w}(\ee)=Id_\VV$. Hence it is always enough to consider multiplication algorithms satisfying $w(\ee)=Id_\VV$.
\end{remark}
Note that for both of the considered multiplication algorithms ($w_1(\xx)=\PP(\xx^{1/2})$ and $w_2(\xx)=t_\xx\in\TT$) we have $w_i(\ee)=Id_\VV$, $i=1,2$

We start with the following crucial lemma.
\begin{lemma}\label{cij}
Let $\ab$ and $\bb$ be two distinct non-orthogonal primitive idempotents in $\En$. Then there exist a primitive idempotent $\cb$, orthogonal to $\ab$, and an element $\zz\in\En\left(\ab,1/2\right)\cap\En\left(\cb,1/2\right)$ such that
\begin{align}\label{bacz}
\bb=\lambda^2 \ab + \mu^2\cb + \lambda\mu\zz,
\end{align}
where $\lambda^2=\scalar{\ab,\bb}$, $\lambda^2+\mu^2=1$ and $\norm{\zz}^2=2$. 
\end{lemma}
\begin{proof}
Recall that subalgebra $\En(\ab,\bb)$ of $\En$ generated by non-orthogonal idempotents $\ab$ and $\bb$ is isomorphic to $S_2(\RR)$ with algebra isomorphism $\rho$ defined by (see \cite[Proposition IV.1.6]{FaKo1994}) 
\begin{align*}
\rho(\alpha a_0+\beta b_0 + \gamma u_0)=\alpha \ab+\beta \bb+\gamma \ub,\quad\mbox{ where } \ub=\ab\bb,
\end{align*}
with $a_0=\begin{pmatrix} 1 & 0 \\ 0 & 0 \end{pmatrix}$, $b_0=\begin{pmatrix} \lambda^2 & \lambda\mu \\ \lambda\mu & \mu^2 \end{pmatrix}$ and $u_0=(a_0\cdot b_0+b_0\cdot a_0)/2$, with $\lambda^2=\scalar{\ab,\bb}$ and $\lambda^2+\mu^2=1$. The value of $\scalar{\ab,\bb}$ is strictly positive due to the positive definiteness of operator $\LL(\bb)$ and
\begin{align*}
\scalar{\ab,\bb}=\scalar{\ab^2,\bb}=\scalar{\ab,\LL(\bb)\ab}.
\end{align*}

Note that
\begin{align*}
c_0:=\frac{b_0+\lambda^2 a_0-2u_0}{\mu^2}=\begin{pmatrix}0 & 0 \\ 0 & 1 \end{pmatrix}.
\end{align*}
Therefore, there exists a primitive idempotent $\cb$, orthogonal to $\ab$, such that
\begin{align}\label{mult}
\cb:=\rho(c_0)=\frac{\bb+\lambda^2 \ab-2\ub}{\mu^2}.
\end{align}
Since $\PP(\dd)$ is the projection onto $\En(\dd,1)=\RR\dd$ for a arbitrary idempotent $\dd$, we have $\PP(\ab)\bb=\lambda^2\ab$ and $\PP(\bb)\ab=\lambda^2\bb$, which can be rewritten as (recall that $\PP(\xx)=2\LL^2(\xx)-\LL\left(\xx^2\right)$)
\begin{align*}
\ab\ub=\frac{\ub+\lambda^2\ab}{2},\qquad\bb\ub=\frac{\ub+\lambda^2\bb}{2}.
\end{align*}
Multiplying \eqref{mult} by $\LL(\bb)$ and $\LL(\cb)$, respectively, we obtain
\begin{align*}
\bb\cb  =\frac{\bb+\lambda^2\ub-2\bb\ub}{\mu^2},\qquad \cb =\frac{\bb\cb-2\cb\ub}{\mu^2},
\end{align*}
from which we deduce that
$$\bb\cb=\frac{\bb+\lambda^2\ub-(\ub+\lambda^2\bb)}{\mu^2}=\bb-\ub$$ and
$$\cb\ub=\frac{\bb\cb-\mu^2\cb}{2}=\frac{\bb-\ub-\mu^2\cb}{2}=\frac{\ub-\lambda^2\ab}{2},$$
where in the latter equality \eqref{mult} was used. Hence, for $\zz=\frac{2}{\lambda\mu}(\ub-\lambda^2\ab)$ we obtain
\begin{align*}
\ab\zz =\frac{2}{\lambda\mu}(\ab\ub-\lambda^2\ab)=\tfrac{1}{2}\zz,\qquad
\cb\zz =\frac{2}{\lambda\mu}\cb\ub=\tfrac{1}{2}\zz.
\end{align*}
Therefore $\zz\in\En\left(\ab,1/2\right)\cap\En\left(\cb,1/2\right)$. It remains to show that equality \eqref{bacz} holds and $\norm{\zz}^2=2$. 

Inserting $\ub=\frac{\lambda\mu}{2}\zz+\lambda^2\ab$ to \eqref{mult} we obtain \eqref{bacz}. By \eqref{xxyy} we have $\zz^2=\frac{\norm{\zz}^2}{2}(\ab+\cb)$. Square both side of equation \eqref{bacz} to we arrive at
\begin{align*}
\bb^2 & =\lambda^4\ab^2+\mu^4\cb^2+\lambda^2\mu^2\frac{\norm{\zz}^2}{2}(\ab+\cb)+2\lambda^2\mu^2\ab\cb+2\lambda^3\mu\ab\zz+2\lambda\mu^3\bb\zz \\
 & = \lambda^2\left(\lambda^2+\mu^2\frac{\norm{\zz}^2}{2}\right)\ab+\mu^2\left(\lambda^2\frac{\norm{\zz}^2}{2}+\mu^2\right)\cb+\lambda\mu(\lambda^2+\mu^2)\zz.
\end{align*}
But $\bb=\bb^2$, so $\norm{\zz}^2=2$.
\end{proof}

The following Lemma is a technical result, which will be extensively used in the proofs of theorems.
\begin{lemma}\label{lema2}
Let $\ab$ and $\bb$ be primitive orthogonal idempotents and let $\zz\in\En\left(\ab,1/2\right)\cap\En\left(\bb,1/2\right)$.
\begin{enumerate}[(i)]
\item Then
\[ \tau_\ab(\zz)\ee=\ee+\zz+\tfrac{\norm{\zz}^2}{2}\bb\mbox{ and }\tau_\ab(\zz)\ab=\ab+\zz+\tfrac{\norm{\zz}^2}{2}\bb.\]
\item Suppose additionally that $\norm{\zz}^2=2$. Then
\begin{align*}
\PP(\alpha \ab+\beta\bb+\gamma\zz)\ab & = \alpha^2\ab+\gamma^2\bb+\alpha\gamma\zz, \\
\PP(\alpha \ab+\beta\bb+\gamma\zz)\zz & = 2 \alpha\gamma\ab+2\beta \gamma\bb+(\alpha\beta+\gamma^2)\zz,
\end{align*}
for $\alpha, \beta, \gamma\in\RR$.
\item If $\xx=\sum_{i=1}^r\alpha_i\cb_i$ for $\alpha_i>0$, and $\yy=\sum_{i=1}^r\beta_i\cb_i$ for a Jordan frame $(\cb_i)_{i=1}^r$, then
\begin{align*}
\PP(\xx^{1/2})\yy=\LL(\xx)\yy.
\end{align*}
\end{enumerate}
\end{lemma}

\begin{proof}
We start with $(i)$. By \eqref{taucb} the Frobenius transformation $\tau_{\ab}(\zz)$ is given by 
\begin{align*}
\tau_{\ab}(\zz)=I+2\zz\Box\ab+2(\zz\Box\ab)^2,
\end{align*}
where $\zz\Box\ab =\tfrac{1}{2}\LL(\zz)+\LL(\zz)\LL(\ab)-\LL(\ab)\LL(\zz)$.
It is easy to see that
\begin{align*}
I_1 & =(\zz\Box\ab)\ee =\tfrac{1}{2}\zz+\LL(\zz)\ab-\LL(\ab)\zz=\tfrac{1}{2}\zz,\\
I_2 & = (\zz\Box\ab)\ab =\tfrac{1}{4}\zz+\LL(\zz)\ab-\tfrac{1}{2}\LL(\ab)\zz=\tfrac{1}{2}\zz,\\
I_3 & =(\zz\Box\ab)\zz =\tfrac{1}{2}\zz^2+\tfrac{1}{2}\zz^2-\LL(\ab)\zz^2=\LL(\ee-\ab)\zz^2.
\end{align*}
By \eqref{xxyy} we obtain that $\zz^2=\tfrac{\norm{\zz}^2}{2}(\ab+\bb)$ and 
\begin{align*}
I_3 = \tfrac{\norm{\zz}^2}{2}\bb.
\end{align*}
Finally
\begin{align*}
\tau_{\ab}(\zz)\ee & =\ee+2I_1+2(\zz\Box\ab)\tfrac{1}{2}\zz=\ee+2I_1+I_3, \\
\tau_{\ab}(\zz)\ab & =\ab+2I_2+2(\zz\Box\ab)\tfrac{1}{2}\zz=\ab+2I_2+I_3.
\end{align*}

Now we prove $(ii)$. Denote $\alpha \ab+\beta\bb+\gamma\zz$ by $\xx$ and recall that $\PP(\xx)=2\LL^2(\xx)-\LL(\xx^2)$. By the definition of $\En\left(\ab,1/2\right)\cap\En\left(\bb,1/2\right)$ it follows that $\LL(\ab)\zz=\LL(\bb)\zz=\tfrac{1}{2}\zz$ and
\begin{align*}
\LL(\xx)\ab =\alpha\ab+\gamma\zz/2.
\end{align*}
Moreover, by \eqref{xxyy} we have $\zz^2=\ab+\bb$ and
\begin{align*}
J_1 &=\LL^2(\xx)\ab  =\LL(\xx)(\alpha\ab+\tfrac{1}{2}\gamma\zz)  = \alpha(\alpha\ab+\tfrac{1}{2}\gamma\zz)+\tfrac{1}{2}\gamma \LL(\xx)\zz = \\
 & =\alpha^2\ab+\tfrac{1}{2}\alpha\gamma\zz+\tfrac{1}{2}\gamma (\tfrac{1}{2}\alpha \zz+\tfrac{1}{2}\beta\zz+\gamma(\ab+\bb)) = 
(\alpha^2+\tfrac{1}{2}\gamma^2)\ab+\tfrac{1}{2}\gamma^2\bb+\tfrac{1}{4}(3\alpha+\beta)\gamma\zz.
\end{align*}
Proceeding similarly we obtain
\begin{align*}
\xx^2 & = (\alpha^2+\gamma^2)\ab+(\beta^2+\gamma^2)\bb+\gamma(\alpha+\beta)\zz,\\
J_2 &=\LL\left(\xx^2\right)\ab =(\alpha^2+\gamma^2)\ab+\tfrac{1}{2}\gamma(\alpha+\beta)\zz,
\end{align*}
and finally
\begin{align*}
\PP(\xx)\ab & = 2J_1-J_2=\alpha^2\ab+\gamma^2\bb+\alpha\gamma\zz.
\end{align*}
The second part of $(ii)$ is proved analogously.

Let us note that if $\xx=\sum_{i=1}^r \alpha_i\cb_i$, then $\xx^{1/2}=\sum_{i=1}^r \sqrt{\alpha_i}\cb_i$ and
\begin{align*}
\LL\left( \sum_{i=1}^r \alpha_i\cb_i\right) \sum_{i=1}^r \beta_i\cb_i = \sum_{i=1}^r\alpha_i\beta_i\cb_i.
\end{align*}
Thus for $\yy=\sum_{i=1}^r \beta_i\cb_i$ we have $\LL^2(\xx^{1/2})\yy=\LL(\xx)\yy$, which is equivalent to the condition $\PP(\xx^{1/2})\yy=\LL(\xx)\yy$.
\end{proof}

\begin{theorem}[$w_1$-logarithmic Cauchy functional equation]\label{detwth}
Let $f\colon \VV\to\RR$ be a function such that
\begin{align}\label{detP}
f(\xx)+f(\yy)=f\left(\PP\left(\xx^{1/2}\right)\yy\right),\quad (\xx,\yy)\in\VV^2.
\end{align}
Then there exists a logarithmic function $H$ such that
\begin{align*}
f(\xx)=H(\det\xx)
\end{align*}
for all $\xx\in\VV$.
\end{theorem}
\begin{proof}
Put $\xx=\alpha \cb+\cbo$ and $\yy=\beta\cb+\cbo$ for $\alpha, \beta>0$ for an idempotent $\cb$, where $\cbo=\ee-\cb$. Then, by $(iii)$ of Lemma \ref{lema2}, we have
\begin{align*}
f(\alpha\cb+\cbo)+f(\beta\cb+\cbo)=f(\alpha\beta\cb+\cbo).
\end{align*}
Function $(0,\infty)\ni\lambda\mapsto f(\lambda\cb+\cbo)\in\RR$ is logarithmic, hence, we may write 
\begin{align*}
f(\lambda\cb+\cbo)=H_\cb(\lambda),\quad \lambda>0,
\end{align*}
where $H_\cb$ is logarithmic. 

Let $\xx=\sum_{i=1}^r \alpha_i\cb_i$, $\alpha_i>0$ for $i=1, \ldots,r$, where $(\cb_i)_{i=1}^r$ is a Jordan frame. Note that $\xx=\prod_{i=1}^r\left(\alpha_i\cb_i+\cbo_i\right)$. Hence, by \eqref{detP} and Lemma \ref{lema2} $(iii)$ we get inductively
\begin{align}\label{start}\begin{split}
f(\xx) & =f\left(\prod_{i=1}^r\left(\alpha_i\cb_i+\cbo_i\right)\right)=f\left(\PP(\alpha_r^{1/2}\cb_r+\cb_r^\bot)\prod_{i=1}^{r-1}\left(\alpha_i\cb_i+\cbo_i\right)\right)\\
 & =f\left(\prod_{i=1}^{r-1}\left(\alpha_i\cb_i+\cbo_i\right)\right)+f(\alpha_r\cb_r+\cb_r^\bot)=\sum_{i=1}^r H_{\cb_i}(\alpha_i).
\end{split}
\end{align}
Our aim is to show that $H_\ab\equiv H_\bb$ for any primitive idempotents $\ab$ and $\bb$. Then $H_\ab\equiv H_\bb\equiv H$ and by \eqref{start} we obtain
$$f(\xx)=\sum_{i=1}^r H(\alpha_i)=H\left(\prod_{i=1}^r\alpha_i\right)=H(\det\xx).$$

Consider any distinct non-orthogonal primitive idempotents $\ab$ and $\bb$. By Lemma \ref{cij} there exists a primitive idempotent $\cb$, orthogonal to $\ab$, and $\zz\in\En\left(\ab,1/2\right)\cap\En\left(\cb,1/2\right)$ such that
\begin{align*}
\bb=\lambda^2 \ab + \mu^2\cb + \lambda\mu\zz,
\end{align*}
where $\lambda^2+\mu^2=1$ and $\norm{\zz}^2=2$.

Without loss of generality we may assume $\lambda>0$.

By \eqref{detP} we have
\begin{align}\label{deteq}
f\left(\PP(\xx)\yy^2\right)=f\left(\xx^2\right)+f(\yy^2)=f\left(\PP(\yy)\xx^2\right)
\end{align}
for any $\xx,\yy\in\VV$. We will choose $\xx$ and $\yy$ such that $\PP(\xx)\yy^2=\alpha\ab+\ab^\bot$ and $\PP(\yy)\xx^2=\alpha\bb+\bb^\bot$ for $\alpha$ in a nonvoid open interval. 

Let
\begin{align*}
\xx & =x_1\ab+x_2\cb+x_3\zz+(\ab+\cb)^\bot,\\
\yy & =y_1\ab+y_2\cb+y_3\zz+(\ab+\cb)^\bot.
\end{align*}
Then
\begin{align*}
\yy^2 & =(y_1^2+y_3^2)\ab+(y_2^2+y_3^2)\cb+y_3(y_1+y_2)\zz+(\ab+\cb)^\bot.
\end{align*}
Using Lemma \ref{lema2} $(ii)$, we arrive at
\begin{align*}
\PP(\xx)\ab & =x_1^2\ab+x_3^2\cb+x_1x_3\zz,\\
\PP(\xx)\cb & =x_3^2\ab+x_2^2\cb+x_2x_3\zz,\\
\PP(\xx)\zz & =2x_1x_3\ab+2x_2x_3\cb+(x_1x_2+x_3^2)\zz,\\
\PP(\xx)(\ab+\cb)^\bot & =(\ab+\cb)^\bot.
\end{align*}
Thus
\begin{align*}
\PP(\xx)\yy^2= & (y_1^2+y_3^2)\left(x_1^2\ab+x_3^2\cb+x_1x_3\zz\right) +(y_2^2+y_3^2)\left(x_3^2\ab+x_2^2\cb+x_2x_3\zz\right) \\ & +
 y_3(y_1+y_2)\left(2x_1x_3\ab+2x_2x_3\cb+(x_1x_2+x_3^2)\zz\right)
 +(\ab+\cb)^\bot = \\
 & =\left((y_1^2 + y_3^2) x_1^2 + (y_2^2 + y_3^2) x_3^2 + 2 y_3 (y_1 + y_2) x_1 x_3\right)\ab \\
 & +\left((y_1^2 + y_3^2) x_3^2 + (y_2^2 + y_3^2) x_2^2 + 2 y_3 (y_1 + y_2) x_2 x_3 \right)\cb \\
 & +\left((y_1^2 + y_3^2) x_1x_3 + (y_2^2 + y_3^2) x_2 x_3 + y_3 (y_1 + y_2) (x_1x_2 + x_3^2) \right)\zz +(\ab+\cb)^\bot.
\end{align*}
By symmetry
\begin{align*}
\PP(\yy)\xx^2 & =\left((x_1^2 + x_3^2) y_1^2 + (x_2^2 + x_3^2) y_3^2 + 2 x_3 (x_1 + x_2) y_1 y_3\right)\ab \\
 & +\left((x_1^2 + x_3^2) y_3^2 + (x_2^2 + x_3^2) y_2^2 + 2 x_3 (x_1 + x_2) y_2 y_3 \right)\cb \\
 & +\left((x_1^2 + x_3^2) y_1y_3 + (x_2^2 + x_3^2) y_2 y_3 + x_3 (x_1 + x_2) (y_1y_2 + y_3^2) \right)\zz +(\ab+\cb)^\bot.
\end{align*}
Let
\begin{align*}
\begin{array}{lll}
x_1 = C \frac{\sqrt{\alpha} \mu^2 \left(\lambda^2+\sqrt{\alpha} \left(1+\lambda^2\right) \mu^2\right)}{1+\lambda^2 \mu^2},
& x_2 = C \left(\frac{1-\sqrt{\alpha}}{1+\lambda^2 \mu^2}-\mu^2\right),
& x_3 =C \sqrt{\alpha} \lambda \mu, \\
y_1 =\lambda^{-2}\mu^{-2},
& y_2 =1,
& y_3 =-\frac{\sqrt{\alpha}+\lambda^2 \mu^2}{\left(1-\sqrt{\alpha}\right) \lambda^3 \mu},
\end{array}
\end{align*}
where $C=\frac{\left(1-\sqrt{\alpha}\right) \lambda^3 \left(1+\lambda^2 \mu^2\right)}{\alpha (1-\lambda^4)+2 \sqrt{\alpha} \lambda^2-\lambda^4\left(1-\mu^4\right)}$. Elements $\xx$ and $\yy$ belong to the cone $\VV$ if $x_1x_2>x_3^2$ and $y_1y_2>y_3^2$ (see \cite[Exercise 7b)]{FaKo1994}). This conditions are satisfied for $\alpha\in\left(0,\frac{\lambda^8}{(1+\lambda^2)^2} \right)$. Thus for $\alpha\in\left(0,\frac{\lambda^8}{(1+\lambda^2)^2} \right)$ we have $\xx,\yy\in\VV$ and
\begin{align*}
\PP(\xx)\yy^2&=\alpha\ab+\cb+(\ab+\cb)^\bot=\alpha\ab+\ab^\bot, \\
\PP(\yy)\xx^2&=(\alpha\lambda^2+\mu^2)\ab+(\alpha\mu^2+\lambda^2)\cb+\lambda\mu(\alpha-1)\zz+(\ab+\cb)^\bot=\alpha\bb+\bb^\bot. 
\end{align*}
Note that there exists infinitely many such pairs $(\xx,\yy)$ - the one presented above is the unique one with $y_1 =\lambda^{-2}\mu^{-2}$ and $y_2=1$.
Inserting this result into \eqref{deteq}, we arrive at
\begin{align*}
H_\ab(\alpha)=H_{\bb}(\alpha)
\end{align*}
for any $\alpha$ in nonvoid open interval. This implies $H_\ab\equiv H_\bb$ for any non-orthogonal primitive idempotents $\ab$ and $\bb$.

If $\ab$ and $\bb$ are orthogonal primitive idempotents, then there exists a primitive idempotent $\cb$, which is non-orthogonal to $\ab$ and $\bb$, so
\begin{align*}
H_\ab\equiv H_\cb\equiv H_\bb.
\end{align*}
Hence, $H_\ab\equiv H_\bb\equiv H$ for arbitrary primitive idempotents $\ab$ and $\bb$, what completes the proof. 
\end{proof}

We will now prove the result for multiplication algorithm connected with triangular group $\TT$. It is closely related to the problem of finding all real characters of $\TT$ - see Remark \ref{character} after the proof.

\begin{theorem}[$w_2$-logarithmic Cauchy functional equation]\label{w2th}
Let $f\colon \VV\to \RR$ be a function satisfying
\begin{align}\label{mainriesz}
f(\xx)+f(\yy)=f(t_{\yy}\xx)
\end{align}
for any $\xx$ and $\yy$ in the cone $\VV$ of rank $r$, $t_{\yy}\in\TT$, where $\TT$ is the triangular group with respect to the Jordan frame $\left(\cb_i\right)_{i=1}^r$. Then there exist generalized logarithmic functions $H_1,\ldots, H_r$ such that for any $\xx\in\VV$,
\begin{align*}
f(\xx)=\sum_{k=1}^r H_k(\Delta_k(\xx)),
\end{align*}
where $\Delta_k$ is the principal minor of order $k$ with respect to $\left(\cb_i\right)_{i=1}^r$.
\end{theorem}
\begin{proof}
Let $(\cb_i)_{i=1}^r$ be a complete system of primitive orthogonal idempotents corresponding to $\TT$. Since $t_\xx=\PP(\xx^{1/2})$ for elements of the form $\xx=\sum_{i=1}^r \alpha_i \cb_i$, as in the previous proof, \eqref{start} is valid. Thus
\begin{align*}
f\left(\sum_{i=1}^r \alpha_i \cb_i\right)=\sum_{i=1}^r H_{\cb_i}(\alpha_i)
\end{align*}
for $\alpha_i>0$, where $H_{\cb_i}(\alpha):=f(\alpha\cb_{i}+\cbo_{i})$, $i=1,\ldots,r$, are generalized logarithmic functions. 
Denote $H_i\equiv H_{\cb_i}-H_{\cb_{i+1}}$, $i=1,\ldots,r-1$ and $H_r\equiv H_{\cb_r}$. Then $H_{\cb_i}\equiv\sum_{k=i}^r H_k$, so
\begin{align*}f\left(\sum_{i=1}^r \alpha_i \cb_i\right)=\sum_{i=1}^r \sum_{k=i}^r H_k(\alpha_i)=\sum_{k=1}^r\sum_{i=1}^kH_k(\alpha_i)=\sum_{k=1}^r H_k\left(\prod_{i=1}^k\alpha_i\right).
\end{align*}

Take now any $\xx\in\VV$. By \eqref{triangular} there exist elements $\zz^{(j)}=\sum_{k=j+1}^r \zz_{kj}$, $\zz_{kj} \in \En_{jk}, 1\leq j\leq r-1$, and positive numbers $\alpha_1, \ldots ,\alpha_r$ such that the triangular decomposition reads
\begin{align*}
\xx=\tau_{\cb_1}(\zz^{(1)})\tau_{\cb_2}(\zz^{(2)})\ldots\tau_{\cb_{r-1}}(\zz^{(r-1)})\left(\sum_{k=1}^r \alpha_k \cb_k \right).
\end{align*}
By \eqref{deltast} we have inductively
\begin{align*}
\Delta_j(\xx)&=\Delta_j\left(\tau_{\cb_1}(\zz^{(1)})\ee\right)\Delta_j\left(\tau_{\cb_2}(\zz^{(2)})\ldots\tau_{\cb_{r-1}}(\zz^{(r-1)})\left(\sum_{k=1}^r \alpha_k \cb_k \right)\right) \\
& =\Delta_j\left(\sum_{k=1}^r \alpha_k \cb_k \right)\prod_{i=1}^{r-1}\Delta_j\left(\tau_{\cb_i}(\zz^{(i)})\ee\right).
\end{align*}
Recall that $\Delta_j\left(\sum_{k=1}^r \alpha_k \cb_k \right)=\prod_{k=1}^j\alpha_k$ and by \eqref{deltastau} we have $\Delta_j\left(\tau_{\cb_i}(\zz^{(i)})\ee\right)=1$ for any $i=1,\ldots,r-1$. Thus 
$$\Delta_j(\xx)=\prod_{k=1}^j\alpha_k.$$
Using \eqref{mainriesz} we arrive at
\begin{align*}
f(\xx)=f\left(\sum_{k=1}^r \alpha_k\cb_k\right)+\sum_{j=1}^{r-1} f\left(\tau_{\cb_j}(\zz^{(j)})\ee\right)=\sum_{k=1}^r H_k(\Delta_k(\xx))+\sum_{j=1}^{r-1} f\left(\tau_{\cb_j}(\zz^{(j)})\ee\right)
\end{align*}
for $\zz^{(j)}\in \bigoplus_{k=j+1}^r \En_{jk}$. Our aim is to show that $f\left(\tau(\zz^{(j)})\ee\right)=0$ for each $1\leq j\leq r-1$. 

For $\alpha^{(j)}=(\alpha_{j+1},\ldots,\alpha_r)\in\RR^{r-j}$ define $\zz^{(j)}(\alpha^{(j)})\in\bigoplus_{k=j+1}^r \En_{jk}$ by
\begin{align*}
\zz^{(j)}(\alpha^{(j)})=\sum_{k=j+1}^r \alpha_k \zz_{jk},
\end{align*}
where $\zz_{ij}$ are the fixed elements from the triangular decomposition of $\xx$. 
Since (it can be quickly shown using \eqref{taucb})
\begin{align*}
\tau_{\cb_j}\left(\zz^{(j)}(\alpha^{(j)})\right)\tau_{\cb_j}\left(\zz^{(j)}(\beta^{(j)})\right)=\exp\left(2\zz^{(j)}(\alpha^{(j)}+\beta^{(j)})\Box\cb_j\right)=\tau_{\cb_j}\left(\zz^{(j)}(\alpha^{(j)}+\beta^{(j)})\right)
\end{align*}
we have
\begin{align*}
f\left(\tau_{\cb_j}\left(\zz^{(j)}(\alpha^{(j)})\right)\ee\right)+f\left(\tau_{\cb_j}\left(\zz^{(j)}(\beta^{(j)})\right)\ee\right)=f\left(\tau_{\cb_j}\left(\zz^{(j)}(\alpha^{(j)}+\beta^{(j)})\right)\ee\right).
\end{align*}
Function $\RR^{r-j}\ni \alpha^{(j)} \mapsto f\left(\tau_{\cb_j}\left(\zz^{(j)}(\alpha^{(j)})\right)\ee\right)\in\RR$ is therefore additive, hence,
\begin{align}\label{ckj}
f\left(\tau_{{\cb_j}}\left(\zz^{(j)}(\alpha^{(j)})\right)\ee\right)=\sum_{k=j+1}^r \Lambda^{(j)}_k (\alpha_k),
\end{align}
where $\Lambda^{(j)}_k$, $j+1\leq k\leq r$, are additive functions on $\RR$. We will show that $\Lambda^{(j)}_k\equiv 0$ for any $1\leq j\leq r-1$, $j+1\leq k\leq r$.

Put now $\xx =\tfrac{1}{\alpha^2}\cb_j+\cb_j^\bot$ for $\alpha>0$. 
Then by \eqref{mainriesz} we have for $1\leq j\leq r-1$, $j+1\leq k\leq r$
\begin{align*}
f\left(t_\xx\tau_{\cb_j}(\zz_{jk})\xx^{-1}\right)=f\left(t_\xx\ee\right)+f\left(\tau_{\cb_j}(\zz_{jk})\ee\right)+f(\xx^{-1}).
\end{align*}
Since $t_\xx\ee=\xx$ and $f(\xx)+f(\xx^{-1})=f(\ee)=0$ we arrive at
\begin{align*}
f\left(\tau_{\cb_j}(\zz_{jk})\ee\right)=f\left(t_\xx\tau_{\cb_j}(\zz_{jk})\xx^{-1}\right).
\end{align*}
We will compute the argument of $f$ on the right hand side of the above equation. By Lemma \ref{lema2} {$(i)$} we obtain
\begin{align*}
\tau_{\cb_j}(\zz_{jk})\xx^{-1} & =\tau_{\cb_j}(\zz_{jk})(\ee+(\alpha^2-1)\cb_j)=\ee+\zz_{jk}+\tfrac{\norm{\zz_{jk}}^2}{2}\cb_k+(\alpha^2-1)(\cb_j+\zz_{jk}+\tfrac{\norm{\zz_{jk}}^2}{2}\cb_k) \\
 & =\xx^{-1}+\alpha^2\tfrac{\norm{\zz_{jk}}^2}{2}\cb_k+\alpha^2\zz_{jk}.
\end{align*}
Note that in this case we have $t_\xx=\PP\left(\xx^{1/2}\right)$. Thus by Lemma \ref{lema2} {$(ii)$} we get $\PP\left(\xx^{1/2}\right)\zz_{jk}=\tfrac{1}{\alpha}\zz_{jk}$ and
\begin{align*}
\PP\left(\xx^{1/2}\right)\tau_{\cb_j}(\zz_{jk})\xx^{-1}=\ee+\alpha^2\tfrac{\norm{\zz_{jk}}^2}{2}\cb_k+\alpha\zz_{jk}.
\end{align*}
Again, by {Lemma} \ref{lema2} {$(i)$}, it is clear that $t_\xx\tau_{\cb_j}(\zz_{jk})\xx^{-1}=\tau_{\cb_j}(\alpha\zz_{jk})\ee$,
what implies that
\begin{align*}
f\left(\tau_{\cb_j}(\zz_{jk})\ee\right)=f\left(\tau_{\cb_j}(\alpha\zz_{jk})\ee\right).
\end{align*}
But
\begin{align*}
\alpha\zz_{jk}=\zz^{(j)}(0,\ldots,0,\underbrace{\alpha}_{k},0,\ldots,0),
\end{align*}
hence, by \eqref{ckj}, we obtain for any $\alpha>0$,
\begin{align*}
\Lambda^{(j)}_k(1)=\Lambda^{(j)}_k(\alpha).
\end{align*}
$\Lambda^{(j)}_k$ is additive and constant, therefore $\Lambda^{(j)}_k\equiv 0$, what completes the proof.
\end{proof}

\begin{remark}\label{character}
Note that equation \eqref{mainriesz} may be rewritten in the form
\begin{align*}
h(s)h(t)=h(s t),\quad(s,t)\in\TT^2,
\end{align*}
where $h(t)=\exp f(t\ee)$ for $t\in\TT$. This way we obtain the form of all real characters of triangular group {(see also \citet{GG2000})}:
\begin{align*}
h(t)=\prod_{k=1}^r M_k(\Delta_k(t\ee)),
\end{align*}
where $M_k$, $k=1,\ldots,r$, are generalized multiplicative functions ($M_k(ab)=M_k(a)M_k(b)$, $a,b>0$).
\end{remark}

\begin{remark}
If we impose on $f$ some mild conditions (eg. measurability) in Theorem \ref{w2th}, there exists $s\in\RR^r$ such that
$$f(\xx)=\log\Delta_s(\xx)$$
for any $\xx\in\VV$.
\end{remark}

{In general we don't know the general form of $w$-logarithmic functions for any multiplication algorithm $w$. But if we assume additionally that $w$-logarithmic function is $K$-invariant ($f(x) = f(kx)$ for any $k\in K$ and $\xx\in\VV$), then we obtain following result.}
\begin{theorem}\label{XXX}
Let $f\colon\VV\to\RR$ be a function satisfying
\begin{align*}
f(\xx)+f(w(\ee)\yy)=f(w(\xx)\yy),\quad(\xx,\yy)\in\VV^2,
\end{align*}
where $w$ is a multiplication algorithm.
Assume additionally that $f$ is $K$-invariant. Then there exists a logarithmic function $H$ such that
$$f(\xx)=H(\det\xx)$$
for any $\xx\in\VV$.
\end{theorem}
\begin{proof}
For any $g\in G$ there exists its polar decompositions, ie. there exist $\zz\in\VV$ and $k\in K$ such that $g=\PP(\zz)k$ (see \cite[Theorem III.5.1]{FaKo1994}). For $g=w(\xx)$ we obtain identity
$w(\xx)=\PP(\zz_\xx)k_\xx$ for $\zz_\xx\in\VV$ and $k_\xx\in K$. Since $w(\xx)\ee=\xx$, we have $\zz_\xx=\xx^{1/2}$.
Thus
\begin{align*}
f(\xx)+f(w(\ee)\yy)=f(w(\xx)\yy)=f(\PP(\xx^{1/2})k_\xx\yy),\quad (\xx,\yy)\in\VV^2.
\end{align*}
Taking $\yy=k_\xx^{-1}\ub$ we arrive at
\begin{align*}
f(\xx)+f(w(\ee) k_\xx^{-1}\ub)=f(\PP(\xx^{1/2})\ub),\quad (\xx,\ub)\in\VV^2.
\end{align*}
But $w(\ee) k_\xx^{-1}\in K$, hence $f(w(\ee) k_\xx^{-1}\ub)=f(\ub)$. This means that functions $f$ is $w_1$-logarithmic. By Theorem \ref{detwth} it then follows that there exists logarithmic function $H$ such that $f(\xx)=H(\det\xx)$ for any $\xx\in\VV$.
\end{proof}

At the end we would like to comment on the form of $w$-logarithmic functions (equation \eqref{wC2}). It turns out that this is more natural form then the following one:
\begin{align*}
f(\xx)+f(\yy)=f(w(\xx)\yy),\quad (\xx,\yy)\in\VV^2,
\end{align*}
what is supported by the following Lemma.

\begin{lemma}[$w$-logarithmic Pexider functional equation]\label{lem1}
Assume that real functions $a$, $b$, $c$ are defined on the cone $\Omega$ and satisfy following functional equation
\begin{align}\label{PexwC}
a(\xx)+b(\yy)=c(w(\xx)\yy),\quad (\xx,\yy)\in\VV^2,
\end{align}
where $w$ is a multiplication algorithm.
Then there exist $w$-logarithmic function $f$ and real constants $a_0$, $b_0$ such that
\begin{align*}
a(\xx) & =f(\xx)+a_0,\\
b(\xx) & =f(w(\ee)\xx)+b_0,\\
c(\xx) & =f(\xx)+a_0+b_0.
\end{align*}
\end{lemma}
\begin{proof}
Inserting into \eqref{PexwC} $\yy=\ee$ and defining $b_0=b(\ee)$ we obtain for $\xx\in\VV$,
\begin{align*}
	a(\xx)=c(\xx)-b_0.
\end{align*}
Inserting it back to \eqref{PexwC} and taking $\xx=\ee$, for $a_0=c(\ee)-b_0$ we have
\begin{align*}
b(\yy)=c(w(\ee)\yy)-a_0.
\end{align*}
Defining $f(\xx)=c(\xx)-a_0-b_0$, we finally get
\begin{align*}
f(\xx)+f(w(\ee)\yy)=f(w(\xx)\yy),\quad (\xx,\yy)\in\VV^2.
\end{align*}
\end{proof}

\section{{Multiplicative functions on the Lorentz Cone}}
Since the Lorentz cone is less known than its matrix colleagues, in this section we will formulate Theorems \ref{detwth} and \ref{w2th} using the language of the Lorentz cone framework. 
One of the main pleasant properties of Lorentz cone is that both multiplication algorithm may be given by explicit formulas. 

Recall that the fifth kind of simple Euclidean Jordan algebra is $\En_L=\RR\times \RR^n$, $n\geq2$. It is convenient to denote elements of $\En_L$ by $\xx=(x_0,x)$, where $x_0\in\RR$ and $x\in\RR^n$. $\En_L$ is endowed with Jordan product \eqref{scL}, which may be written as
\begin{align*}
\xx\yy=(\scalar{\xx,\yy},x_0y+y_0x)\in\RR\times\RR^n,
\end{align*}
where $\scalar{\xx,\yy}=x_0y_0+\scalar{x,y}_n$ and $\scalar{x,y}_n=\sum_{i=1}^n x_iy_i$ is an inner product of $\RR^n$. The norm defined by $\scalar{\cdot,\cdot}_n$ is denoted by $\norm{\cdot}_n$. The neutral element of $\En_L$ is denoted by $\ee=(1,\emptyset)$, where $\emptyset$ is the zero of $\RR^n$.

It can be shown that (see for example \citet{Wes2007L})
\begin{align*}
\PP(\xx^{1/2})\yy=\left(\scalar{\xx,\yy},\sqrt{\det\xx}\,\,y+(y_0+\frac{\scalar{x,y}_n}{x_0+\sqrt{\det\xx}})x\right),
\end{align*}
where the determinant of elements from $\En_L$ takes the form $\det\xx=x_0^2-\norm{x}_n^2$.
The symmetric cone corresponding to $\En_L$ is called the Lorentz cone and is given by $\VV_L=\left\{\xx\in\En_L\colon x_0>\norm{x}_n\right\}$. Following result was previously proved using direct calculations on the Lorentz cone in \citet{Wes2007L}. 
\begin{corollary}[$w_1$-logarithmic Cauchy functional equation on the Lorentz cone]
Let $f\colon \VV_L\to\RR$ be a function such that
\begin{align*}
f(x_0,x)+f(y_0,y)=f\left(\scalar{\xx,\yy},\sqrt{\det\xx}\,\,y+\left(y_0+\frac{\scalar{x,y}_n}{x_0+\sqrt{\det\xx}}\right)x\right), 
\end{align*}
for any $\xx=(x_0,x),\yy=(y_0,y)\in\VV_L$.
Then there exists a logarithmic function $H$ such that
\begin{align*}
f(\xx)=H\left(x_0^2-\norm{x}_n^2\right)
\end{align*}
for all $\xx=(x_0,x)\in\VV_L$.
\end{corollary}

We are going now to give the specification of triangular group $\TT$ to Lorentz cone case. Rank $r$ of $\VV_L$ is $2$ and any idempotent on $\En_L$ is of the form $\cb_u=\tfrac12(1,u)$ for some $u\in\RR^n$ with $\norm{u}_n=1$. Thus any Jordan frame consists of two elements $(\cb_1,\cb_2)=(\cb_u,\cbo_u)$, where $\cbo_u=\ee-\cb_u$. Recall that $\En_{12} = \En_L\left(\cb_1,\tfrac{1}{2}\right) \cap \En_L\left(\cb_2,\tfrac{1}{2}\right)$ consists of elements $\zz=(z_0,z)$ such that
\begin{align*}
\tfrac{1}{2}\zz & =\LL(\cb_1)\zz =\cb_u\zz,\\
\tfrac{1}{2}\zz & =\LL(\cb_2)\zz =(\ee-\cb_u)\zz.
\end{align*}
It may by easily shown that $\cb_u\zz=\tfrac{1}{2}\zz$ if and only if $z_0=0$ and $\scalar{z,u}_n=0$.
Thus 
\begin{align*}
\En_L\left(\cb_1,\tfrac{1}{2}\right) \cap \En_L\left(\cb_2,\tfrac{1}{2}\right)=\En_L\left(\cb_u,\tfrac{1}{2}\right)=\left\{\zz=(0,z)\in\En_L\colon \scalar{z,u}_n=0 \right\}.
\end{align*}
We are ready to define the triangular group $\TT_L$ with respect to fixed Jordan frame $(\cb_u,\cbo_u)$:
\begin{align*}
\TT_L=\left\{\tau_{\cb_u}(\zz)\PP\left(\alpha_1\cb_u+\alpha_2\cbo_u\right)\colon \alpha_1, \alpha_2>0, z_0=0, \scalar{z,u}_n=0\right\}.
\end{align*}
Note that $\TT_L$ is parametrized by $u\in\RR^n$ and this is equivalent to the choice of basis for Cholesky decomposition of matrices. 

In order to define multiplication algorithm $w_2$ on the Lorentz cone we have to find the unique $t_\yy\in\TT_L$ such that $t_\yy\ee=\yy\in\VV_L$. Using formulas obtained in Lemma \ref{lema2} $(i)$, it can by shown that 
$\tau_{\cb_u}(\zz)\PP(\alpha_1\cb_u+\alpha_2\cbo_u)\ee=\yy\in\VV_L$ if and only if
\begin{align}\label{azz}\begin{split}
\alpha_1^2 & =y_0+\scalar{y,u}_n,\\
\alpha_2^2 & =\frac{\det\yy}{y_0+\scalar{y,u}_n},\\
\zz & =(0,z)=\left(0,\frac{y-\scalar{y,u}_nu}{y_0+\scalar{y,u}_n}\right).
\end{split}\end{align}
Note that $\alpha_1^2=y_0+\scalar{y,u}_n=\Delta_1(\yy)$, where $\Delta_1$ is the principal minor of order 1 with respect to the Jordan frame $(\cb_u,\cbo_u)$.
The multiplication algorithm $w_2$ is then defined by
\begin{align*}
w_2(\yy)\xx=t_\yy\xx=\tau_{\cb_u}(\zz)\PP(\alpha_1\cb_u+\alpha_2\cbo_u)\xx,
\end{align*}
where $\alpha_1, \alpha_2>0$ and $\zz$ are as in \eqref{azz}. Carefully using \citet[VI.3.1]{FaKo1994} it can be shown that
\begin{align*}
t_\yy\xx=\sqrt{\det\yy}\,\,\xx+ \Delta_1(\xx)\yy-\sqrt{\det\yy}\,\,\Delta_1(\xx)\cb_u+h_u(\xx,\yy)\cbo_u,
\end{align*}
where $h_u(\xx,\yy)=\frac{2\sqrt{\det\yy}}{\Delta_1(\yy)}\left(\scalar{x,y}_n-\scalar{x,u}_n\scalar{y,u}_n\right)-\left(1-\frac{2 \det\yy}{\Delta_1(\yy)}\right)\scalar{x,u}_n-x_0$.
Thus, by Theorem \ref{w2th} we obtain the following
\begin{corollary}[$w_2$-logarithmic Cauchy functional equation on the Lorentz cone]
Let $f\colon \VV_L\to\RR$ be a function such that
\begin{align*}
f(\xx)+f(\yy)=f\left(\sqrt{\det\yy}\,\,\xx+ \Delta_1(\xx)\yy-\sqrt{\det\yy}\,\,\Delta_1(\xx)\cb_u+h_u(\xx,\yy)\cbo_u\right)
\end{align*}
for any $(\xx,\yy)\in\VV_L^2$, $t_\yy\in\TT_L$, where $\TT_L$ is the triangular group with respect to the Jordan frame $(\cb_u,\cbo_u)$ and $\norm{u}_n=1$.
Then there exist logarithmic functions $H_1$ and $H_2$ such that for all $\xx=(x_0,x)\in\VV_L$,
\begin{align*}
f(\xx)=H_1(x_0+\scalar{x,u}_n)+H_2(x_0^2-\norm{x}_n^2).
\end{align*}
\end{corollary}
Note that although formula for $t_\yy\xx$ may seem very complicated, the situation for the Lorentz cone is far better than for matrix cones of size $n\times n$, where explicit formulation of Cholesky decomposition is much more complex and only recursive algorithms are at hand.

\subsection*{Acknowledgement} This research was partially supported by NCN grant No. 2012/05/B/ST1/00554. The author thanks J. Weso{\l}owski for helpful comments and discussions.

\bibliographystyle{plainnat}

\bibliography{Bibl}

\end{document}